\documentclass{amsart}
\usepackage{amsmath,amsthm,amssymb,amsfonts,graphicx,mathrsfs,bm,amscd,latexsym,color,hyperref}
\newtheorem{theorem}{Theorem}[section]
\newtheorem{lem}[theorem]{Lemma}
\newtheorem{cor}[theorem]{Corollary}

\theoremstyle{definition}

\theoremstyle{remark}
\newtheorem{remark}[theorem]{Remark}

\def\R{\mathbb{R}}
\def\Z{\mathbb{Z}}
\def\N{\mathbb{N}}

\def\cB{\mathcal{B}}
\def\S{\mathbb{S}}

\def\g{\gamma}
\def\d{\delta}

\def\a{\alpha}

\def\be{\textbf{e}}

\DeclareMathOperator{\diam}{diam}
\DeclareMathOperator{\dist}{dist}

\begin{document}

\title{Quasisymmetric extension on the real line}

\author{Vyron Vellis}
\address{Department of Mathematics and Statistics, P.O. Box 35 (MaD), FI-40014 University of Jyv\"askyl\"a, Jyv\"askyl\"a, Finland}
\curraddr{Department of Mathematics and Statistics, P.O. Box 35 (MaD), FI-40014 University of Jyv\"askyl\"a, Jyv\"askyl\"a, Finland}
\email{vyron.v.vellis@jyu.fi}
\thanks{The author was supported by the Academy of Finland project 257482.}

%    General info

\date{\today}


\begin{abstract}
We give a geometric characterization of the sets $E\subset \R$ for which every quasisymmetric embedding $f: E \to \R^n$ extends to a quasisymmetric embedding $f:\R\to\R^N$ for some $N\geq n$.
\end{abstract}

\maketitle

\section{Introduction}\label{sec:intro}

Suppose that $E$ is a subset of a metric space $X$ and $f$ is a quasisymmetric embedding of $E$ into some metric space $Y$. When is it possible to extend $f$ to a quasisymmetric embedding of $X$ into $Y'$ for some metric space $Y'$ 
containing $Y$? Questions related to quasisymmetric extensions have been considered by Beurling and Ahlfors \cite{BA}, Ahlfors \cite{Ah,Ahl}, Carleson \cite{CA}, Tukia and V\"ais\"al\"a \cite{TukVais} and Kovalev and Onninen 
\cite{Kovalev}. 

Tukia and V\"ais\"al\"a \cite{TuVaext2} showed that for $M=\R^p,\S^p$, any quasisymmetric mapping $f:M \to \R^{n}$, with $n>p$, extends to a quasisymmetric homeomorphism of $\R^{n}$ when $f$ is locally close to a similarity. Later, 
V\"ais\"al\"a \cite{Vaisext} extended this result to all compact, co-dimension $1$, $C^1$ or piecewise linear manifolds $M$ in $\R^n$.

In this article we are concerned with the case $X=\R$ and $Y=\R^n$. Specifically, given a set $E \subset \R$ and a quasisymmetric embedding $f$ of $E$ into $\R^n$, we ask when is it possible to extend $f$ to a quasisymmetric 
embedding of $\R$ into $\R^N$ for some $N\geq n$. While any bi-Lipschitz embedding of a compact set $E\subset \R$ into $\R^n$ extends to a bi-Lipschitz embedding of $\R$ into $\R^N$ for some $N\geq n$ \cite{DS}, the same is not true 
for quasisymmetric embeddings. In fact, there exists $E \subset \R$ and a quasisymmetric embedding $f:E \to \R$ that can not be extended to a quasisymmetric embedding $F : \R \to \R^N$ for any $N$; see e.g. \cite[p. 89]{Heinonen}. 
Thus, more regularity for sets $E$ should be assumed.

Following Trotsenko and V\"ais\"al\"a \cite{TroVa}, a metric space $X$ is termed \emph{$M$-relatively connected} for some $M>1$ if, for any point $x\in X$ and any $r>0$ with $\overline{B}(x,r) \neq X$, either 
$\overline{B}(x,r) = \{x\}$ or $\overline{B}(x,r) \setminus B(x,r/M) \neq \emptyset$. A metric space $X$ is called relatively connected if it is $M$-relatively connected for some $M\geq 1$.

With this terminology, our main theorem is stated as follows.

\begin{theorem}\label{thm:main2}
If $E\subset \R$ is $M$-relatively connected and $f: E \to \R^n$ is $\eta$-quasisymmetric then $f$ extends to an $\eta'$-quasisymmetric embedding $F:\R\to \R^{n+n_0}$ where $n_0$ depends only on $M$ and $\eta$ while $\eta'$ depends 
only on $M$, $\eta$ and $n$.
\end{theorem}

%The statement of the theorem is quantitative: if $E$ is $M$-relatively connected and $f$ is $\eta$-quasisymmetric then $n_0$ depends only on $n$, $M$ and $\eta$, and the extension is $\eta'$-quasisymmetric with $\eta'$ depending 
%only on $\eta$ and $M$.

On the other hand, it follows from a theorem of Trotsenko and V\"ais\"al\"a \cite{TroVa} that if $E\subset \R$ is not relatively connected, then there exists a quasisymmetric mapping $f:E \to\R$ that admits no quasisymmetric 
extension $F : \R \to\R^N$ for any $N\geq 1$; see Corollary \ref{cor:relconnec}.

A subset $E$ of a metric space $X$ is said to have the \emph{quasisymmetric extension property in $X$} if every quasisymmetric mapping $f:E \to X$ that can be extended homeomorphically in $X$ can also be extended 
quasisymmetrically in $X$. The question of characterizing such sets $E$, given a space $X$, poses formidable difficulties due to the topological complexity of $X$. For instance, $\S^1$ and $\R$ have the quasisymmetric extension 
property in $\R^2$ \cite{Ah}, but it is unknown whether $\S^n$ or $\R^n$ have this property in $\R^{n+1}$ when $n\geq 2$.

The sets $E \subset \R$ that have the quasisymmetric extension property in $\R$ are characterized by the relative connectedness.

\begin{theorem}\label{thm:main}
A set $E\subset \R$ has the quasisymmetrc extension property in $\R$ if and only if it is relatively connected.
\end{theorem}

The arguments used in the proof of Theorem \ref{thm:main} apply verbatim in the case $X= \S^1$ and $E\subset \S^1$. Thus, if $X$ is quasisymmetric homeomorphic to either $\R$ or $\S^1$, then a set $E \subset X$ has the quasisymmetric
extension property in $X$ if and only if $E$ is relatively connected.

In dimensions $n\geq 2$, however, Theorem \ref{thm:main} fails even for small sets such as the Cantor sets. In Section \ref{sec:examples} we show that for each $n\geq 2$, there exists a relatively connected Cantor set 
$E \subset \R^n$ and a bi-Lipschitz mapping $f : E \to \R^n$ which admits a homeomorphic extension in $\R^n$, but not a quasisymmetric extension in $\R^n$; see Remark \ref{rem:Cantor}.

\subsection*{Acknowledgements} The author wishes to thank Pekka Koskela for bringing this problem to his attention, Tuomo Ojala for various discussions and the anonymous 
referee for valuable comments on the manuscript. 

%The main focus of this paper is the proof of Theorem \ref{main2}. We construct the extension $F$ in Section \ref{sec:ext} and in Section \ref{sec:proof} we show that $F$ is quasisymmetric.

%over this subject.

\section{Preliminaries}\label{sec:prelim}

%We review in this section the notions of quasisymmetry and relative connectedness. 
In the following, given an open bounded interval $I=(a,b)\subset \R$, we denote by $|I|$ its length $b-a$; if $I=\emptyset$ then $|I| = 0$. As usual, $a \vee b$ and $a\wedge b$ denote the maximum and minimum, respectively, of two 
real numbers $a$ and $b$. Finally, for two points $x,y\in\R^n$, we denote by $[x,y]$ the line segment in $\R^n$ with endpoints $x$ and $y$.

\subsection{Mappings}

A homeomorphism $f\colon (X,d) \to (Y,d')$ between two metric spaces is called $L$-\emph{bi-Lipschitz} for some $L>1$ if both $f$ and $f^{-1}$ are $L$-Lipschitz.

A mapping $f\colon (X,d) \to (Y,d')$ is called $\eta$-quasisymmetric if there exists a homeomorphism $\eta \colon [0,+\infty) \to [0,+\infty)$ such that for any $x,a,b \in X$ with $x\neq b$ we have
\[ \frac{d'(f(x),f(a))}{d'(f(x),f(b))} \leq \eta\left ( \frac{d(x,a)}{d(x,b)} \right ).\]
It is a simple consequence of the definition that the composition of a similarity mapping of $\R^n$ and an $\eta$-quasisymmetric mapping between sets of $\R^n$ is $\eta$-quasisymmetric.

If $f$ is $\eta$-quasisymmetric with $\eta(t) = C(t^{\a}\vee t^{1/\a})$ for some $\a\in(0,1]$ and $C>0$ then $f$ is termed \emph{power quasisymmetric} and we say that $f$ is $(C,\a)$-quasisymmetric. An important property of power 
quasisymmetric mappings is that they are bi-H\"older continuous on bounded sets \cite[Corollary 11.5]{Heinonen}.

\begin{lem}\label{lem:Holder}
Suppose that $(X,d)$ is a bounded metric space and $f : (X,d) \to (Y,d')$ is $(C,\a)$-quasisymmetric. There exists $C' >1$ depending only on $C$, $\a$, $\diam{X}$ and $\diam{f(X)}$ such that for all $x,y \in E$,
\[ (C')^{-1} d(x,y)^{1/\a} \leq d'(f(x),f(y)) \leq C' d(x,y)^{\a}. \]
\end{lem}

For doubling connected metric spaces it is known that the quasisymmetric condition is equivalent to a weaker (but simpler) condition known in literature as \emph{weak quasisymmetry}. 
%The proof of the following lemma is similar to that of \cite[Theorem 10.19]{Heinonen}.

\begin{lem}[{\cite[Theorem 10.19]{Heinonen}}]\label{lem:ext}
Let $I\subset \R$ be an interval and $f:I \to \R^n$ be an embedding for which there exists $H \geq 1$ such that for all $x,y,z \in I$
\begin{equation}\label{eq:weakQS}
|x-y|\leq|x-z| \text{ implies } |f(x)-f(y)|\leq H|f(x)-f(z)|.
\end{equation} 
Then $f$ is $\eta$-quasisymmetric with $\eta$ depending only on $H$ and $n$.
\end{lem}

The next lemma is an immediate corollary to Lemma \ref{lem:ext}.
%We conclude the discussion on quasisymmetric mappings with the following corollary to Lemma \ref{lem:ext}.

\begin{lem}\label{lem:jointqs}
Let $I_1,I_2$ be open bounded intervals and $f:I_1\cup I_2 \to \R$ be an embedding. Suppose that there exists $C>1$ such that $|I|/|J| < C$ for all $I,J \in \{I_1,I_2,I_1\cap I_2\}$. If $f|I_1$ and $f|I_2$ are $\eta$-quasisymmetric
then $f|(I_1\cup I_2)$ is $\eta'$-quasisymmetric for some $\eta'$ depending on $\eta$ and $C$.
\end{lem}

\begin{proof}
If $I_1 \subset I_2$ or $I_2 \subset I_1$ there is nothing to prove. Suppose that $I_1 = (a_1,b_1)$, $I_2 = (a_2,b_2)$ with $a_1< a_2 < b_1< b_2$ and denote by $m$ the center of $I_1\cap I_2$. We show that $f|(I_1\cup I_2)$ 
satisfies (\ref{eq:weakQS}). Let $x,y,z\in I_1\cup I_2$ with $|x-y|\leq |x-z|$. Since $f|I_j$ is monotone for each $j=1,2$, $f| I_1\cup I_2$ is monotone and we may assume that either $y<x<z$ or $z<x<y$. 
Assume the first; the second case is identical
%We only work for the first case; the other is identical. 

If all three points are in the same $I_j$ there is nothing to prove. Hence, we may assume that $y\leq a_2$ and $z\geq b_1$.

If $x\leq m$ then $|f(x)-f(y)| \leq \eta(\frac{|x-y|}{|x-b_1|})|f(x)-f(z)| \leq \eta(2C)|f(x)-f(z)|$. 

If $x\geq m$ then $|f(x)-f(y)| = |f(x)-f(a_2)| + |f(a_2)-f(y)| 
                               \leq |f(x)-f(a_2)|(1 + \frac{|f(a_2)-f(a_1)|}{|f(a_2)-f(m)|}) 
                               \leq |f(x)-f(a_2)| (1+ \eta(\frac{|a_1-a_2|}{|a_2-m|})) 
                               \leq (1+\eta(2C))|f(x)-f(a_2)| 
                               \leq (1+\eta(2C))\eta(\frac{|x-a_2|}{|x-z|})|f(x)-f(z)| 
                               \leq (1+\eta(2C))\eta(1)|f(x)-f(z)|$
where for the last inequality we used $|x-a_2| \leq |x-y|\leq |x-z|$.
\end{proof}
%If $y\in I_2$ then $x,z\in I_2$ and $|f(x)-f(y)|\leq \eta(\d_0)|f(x)-f(z)|$. Similarly if 
%$x,y,z\in I_1$. If $y \leq a_2$ and $z\geq b_1$ then the assumption for the relative distances of $y,x,z$ yields $x<m$. Thus, 
%$|f(x)-f(y)| \leq \eta(1)|f(x)-f(a_1)| \leq \eta(1)^2 |f(a_1)-f(m)| \leq \eta(1)^2\eta(2C)|f(m)-f(b_1)| \leq \eta(1)^3\eta(2C)|f(x)-f(b_1)| \leq \eta(1)^4\eta(2C)|f(x)-f(z)|$.

\subsection{Relatively connected sets}\label{sec:relcon}
Relatively connected sets were first introduced by Trotsenko and V\"ais\"al\"a \cite{TroVa} in the study of spaces for which every quasisymmetric mapping is power quasisymmetric. The definition
% of relative connectedness 
given in \cite{TroVa} is equivalent to the one in Section \ref{sec:intro} quantitatively \cite[Theorem 4.11]{TroVa}. 

Relative connectedness is a weak form of the well known notion of uniform perfectness. A metric space $X$ is \emph{$c$-uniformly perfect} for some $c>1$ if for all $x\in X$, $\overline{B}(x,r)\neq X$ implies 
$\overline{B}(x,r)\setminus B(x,r/c) \neq \emptyset$. The difference between the two notions is that relatively connected sets allow isolated points. In particular, if $E$ is $c$-uniformly perfect, then it is $M$-relatively 
connected for all $M>c$, and if $E$ is $M$-relatively connected and has no isolated points, then it is $(2M+1)$-uniformly perfect \cite[Theorem 4.13]{TroVa}. 

The connection between relative connectedness and power quasisymmetric mappings is illustrated in the following theorem from \cite{TroVa}.

\begin{theorem}[{\cite[Theorem 6.20]{TroVa}}]\label{thm:TroVa}
A subset $E$ of a metric space $X$ is relatively connected if and only if every quasisymmetric map $f:E \to X$ is power quasisymmetric. 
\end{theorem}

The necessity of relative connectedness for extensions of quasisymmetric mappings on $\R$ follows now as a corollary. 

\begin{cor}\label{cor:relconnec}
If $E\subset\R$ is not relatively connected, then there exists a monotone quasisymmetric mapping $f:E \to \R$ such that, for every metric space $Y$ containing the Euclidean line $\R$, there exists no quasisymmetric extension 
$F:\R \to Y$ of $f$.
\end{cor}

\begin{proof}
By \cite[Theorem 6.6]{TroVa}, there exists a quasisymmetric mapping $f: E\to \R$ that is not power quasisymmetric. A close inspection in its proof reveals, moreover, that the mapping $f$ is increasing. Let now $Y$ be a metric space 
containing the Euclidean line $\R$. If there was a quasisymmetric extension $F \colon \R \to Y$, then, by Theorem \ref{thm:TroVa}, $F$ would be power quasisymmetric. Thus, $f$ would be power quasisymmetric which is a contradiction.
\end{proof}

\subsection{Relative distance}

Let $E,F$ be two compact sets in a metric space $(X,d)$ both of which contain at least two points. The \emph{relative distance} of $E$ and $F$ is defined to be the quantity
\[ d^*(E,F) = \frac{\dist(E,F)}{\diam{E}\wedge\diam{F}}\]
where $\dist(E,F) = \min\{d(x,y):x\in E, y\in F\}$. 

Note that if $E,F \subset \R^n$ and $f : \R^n \to \R^n$ is a similarity then $d^*(f(E),f(F)) = d^*(E,F)$. In general, if $f: E\cup F \to Y$ is $\eta$-quasisymmetric, then
\begin{equation}\label{eq:relQS}
\frac12 \phi(d^*(E,F)) \leq d^*(f(E),f(F)) \leq \eta(2d^*(E,F))
\end{equation}
where $\phi(t)=(\eta(t^{-1}))^{-1}$; see for example \cite[p. 532]{Tyson}.

%\begin{proof}
%Let $x,z\in E$ and $y \in F$. Then, 
%\[ \frac{|f(x)-f(z)|}{|f(x)-f(y)|} \leq \eta\left ( \frac{|x-z|}{|x-y|} \right ) \leq  \eta\left ( \frac{\diam{E}}{\dist(E,F)} \right ).\]
%Fix $x\in E$ and $y\in E$ such that $|f(x)-f(y)| = \dist(f(E),f(F))$. Let $z\in E$ such that $|f(x)-f(z)| \geq \frac{1}{2}\diam{f(E)}$. Then,
%\[ \frac{\dist(E,F)}{\diam{E}} \leq \phi\left ( \frac{\dist(f(E),f(F))}{\diam{f(E)}} \right ) \leq \phi( d^*(f(E),f(F)) )\]
%with $\phi(t) = (\eta^{-1}((2t)^{-1}))^{-1}$. Applying the same procedure on $F$, $d^*(E,F) \leq d^*(f(E),f(F))$.
%\end{proof}
 
The following remark ties the notions of uniform perfectness in $\R$ and relative distance of sets in $\R$. 

\begin{remark}\label{rem:rel}
A closed set $E\subset \R$ is $c$-uniformly perfect for some $c\geq 1$ if and only if there exists $C>0$ such that for all bounded components $I,J$ of $\R\setminus E$, $d^*(I,J) \geq C$. The constants $c$ and $C$ are quantitatively 
related.
\end{remark}

\section{Quasisymmetric extension on $\R$}\label{sec:ext}
Suppose that $E \subset \R$ is relatively connected and $f:E \to \R^n$ is quasisymmetric. If $E$ is a singleton then trivially $f$ admits a quasisymmetric extension. Moreover, since quasisymmetric functions have a quasisymmetric 
extension to the closure of their domains, we may assume that $E$ is closed. 

In Section \ref{sec:reduc} we construct a quasisymmetric extension $f_0 : E_0 \to \R^m$ of $f$, where $E \subset E_0 \subset \R$ is a uniformly perfect set with no lower or upper bound and $m$ is either $n$ or $n+1$. In Section 
\ref{sec:bridges}, for some $n_0\in\N$ depending only on $M$ and $\eta$, we construct a homeomorphic extension $F_0 : \R \to \R^{n+n_0}$ of $f_0$. Finally, in Section \ref{sec:reflect} we construct a quasisymmetric extension 
$F:\R\to F_0(\R) \subset \R^{n+n_0}$ of $f_0$.

For the rest, $\textbf{0}$ denotes the origin of $\R^n$ and, for each $i=1,\dots,n$, $\be_i$ denotes the vector in $\R^n$ whose $i$-th coordinate is $1$ and the rest are $0$.

\subsection{Two preliminary extensions}\label{sec:reduc}
Throughout this section we assume that $E$ is an $M$-relatively connected closed set and $f$ is an $\eta$-quasisymmetric embedding of $E$ into $\R^n$ with $\eta = C(t^{\a}\vee t^{1/\a})$.

Suppose that $E$ is bounded from above or bounded from below. Then one of the following cases applies. 

\emph{Case 1}. Suppose that $E$ has a lower bound but no upper bound. Applying suitable similarities we may assume that $1\in E$, $\min{E} = 0$ and $f(0)=\textbf{0}$. Let $C_0 = \max\{2,1/\eta^{-1}(1/2)\}$. Set $a_0=0$
and, by relative connectedness, there exists a sequence $\{a_k\}_{k\in\N}\subset E$ with $a_1 = 1$ and $a_k/a_{k-1} \in [C_0,MC_0]$. Set $\tilde{E} = E \cup\{-a_k\}_{k\in\N}$ and $\tilde{f} : \tilde{E} \to \R^{n+1}$ with 
$\tilde{f}|E = f\times\{0\}$ and $\tilde{f}(-a_k) = \{\textbf{0}\}\times \{-|f(a_k)|\}$. 

\emph{Case 2}. Suppose that $E$ has an upper bound but no lower bound. Applying suitable similarities we may assume that $1\in E$, $\max{E} = 0$ and $f(0)=\textbf{0}$. We define $\tilde{E}$ and $\tilde{f}$ similarly to 
Case 1.

\emph{Case 3}. Suppose that $E$ is bounded. Applying suitable similarities, we may assume that $\min E = 0$, $\max E =1$, $\max_{x\in E}|f(x)| =1$ and $\diam{f(E)} =1$. For any $k\in \Z$ define 
$\tilde{E}_k = \{2k+ x \colon x \in E\}$, $\tilde{E} = \bigcup_{k\in\Z}\tilde{E}_k$ and $\tilde{f} : \tilde{E} \to \R^n$ with $\tilde{f}(2k+x) = 2k\be_1+f(x)$. A similar extension in the case $n=1$ has been considered by Lehto and 
Virtanen in \cite[II.7.2]{LeVi}.

\begin{lem}\label{lem:unbound}
In each case, $\tilde{E}$ is an $\tilde{M}$-relatively connected closed set and $\tilde{f}$ is $\tilde{\eta}$-quasisymmetric with $\tilde{M}$ and $\tilde{\eta}$ depending only $M$ and $\eta$.
\end{lem}

\begin{proof}
We only prove the lemma for Case 1 and Case 3; the proof for Case 2 is similar to that of Case 1.

\emph{Case 1}. Note first that $\{-a_n\}_{n\in\N}$ is $M_1$-relatively connected for some $M_1$ depending only on $M$ and $\eta$. Let $x\in \tilde{E}$ and $r>0$ such that $\overline{B}(x,r)\cap\tilde{E} \neq \{x\}$. If $x\in E$ then 
$\overline{B}(x,r)\cap E \neq \{x\}$ and $(\overline{B}(x,r)\setminus B(x,r/M))\cap\tilde{E} \neq \emptyset$. If $x=-a_n$, $n\geq 1$, then $\overline{B}(x,r)\cap \{-a_n\}_{n\in\N} \neq \{x\}$ and 
$(\overline{B}(x,r)\setminus B(x,r/M_1))\cap\tilde{E} \neq \emptyset$. Thus, $\tilde{E}$ is $(M\vee M_1)$-relatively connected.

For the quasisymmetry of $\tilde{f}$, note first that $\tilde{f}$ restricted on $\{-a_n\}_{n\in\N}$ is $C\eta$-quasisymmetric for some $C>1$ depending only on $\eta$. Let $x,y,z \in \tilde{E}$. If all three of them are in $E$ or in 
$\tilde{E}\setminus E$ then the quasisymmetry of $\tilde{f}$ follows trivially.

Assume first that $x,z \in E$ and $y = -a_n$ for some $a_n \in E$. Then, $|\tilde{f}(y)| = |f(a_n)|$, $|y| = |a_n|$ and
\begin{align*}
\frac{|\tilde{f}(x)-\tilde{f}(y)|}{|\tilde{f}(x)-\tilde{f}(z)|} &\leq 2\frac{|f(x)|}{|f(x)-f(z)|} + \frac{|f(x)-f(a_n)|}{|f(x)-f(z)|} \\
                                &\leq 2\eta \left ( \frac{|x|}{|x-z|} \right ) +  \eta\left ( \frac{|x-a_n|}{|x-z|}\right ) \leq 3\eta\left (\frac{|x-y|}{|x-z|}\right ). 
\end{align*}
We work similarly if $x,z \in \{-a_n\}_{n\in\N}$ and $y\in E$.

Assume now that $z\in E$ and $y,x\not\in E$. Let $n_0$ be the smallest integer $n$ such that $a_{n}\geq z$ and set $\overline{z} = -a_{n_0}$. Then, there exist constants $C_1,C_2>1$ depending only on $M$, $C$ and $\a$ such that
\begin{align*}
\frac{|\tilde{f}(x)-\tilde{f}(y)|}{|\tilde{f}(x)-\tilde{f}(z)|} &\leq C_1\min \left \{ \frac{|\tilde{f}(x)-\tilde{f}(y)|}{|\tilde{f}(x)-\tilde{f}(\overline{z})|}, \frac{|\tilde{f}(x)-\tilde{f}(y)|}{|\tilde{f}(x)|} \right \} \\
                                &\leq C_1\min \left \{\eta \left ( \frac{|x-y|}{|x-\overline{z}|} \right ), \eta \left ( \frac{|x-y|}{|x|}\right ) \right \} \leq C_2\eta\left (\frac{|x-y|}{|x-z|}\right ).
%\frac{|f(x)-f(a)|}{\max\{|f(x)-f(m)|, |f(m)-f(b')|\}}
\end{align*}
We work similarly if $z\in \{-a_n\}_{n\in\N}$ and $x,y \in E$.

\emph{Case 3}. We first show that $\tilde{E}$ is $M'$-relatively connected with $M'=8M$. Let $x\in \tilde{E}$ and $r>0$ such that $\overline{B}(x,r) \cap \tilde{E} \neq \{x\}$. Since $\tilde{E}$ is unbounded,
 $\tilde{E}\setminus \overline{B}(x,r) \neq \emptyset$. By periodicity of $\tilde{E}$, we may assume that $x\in E$. If $r\geq 4$ then $\overline{B}(x,r)\setminus B(x,r/2)$ contains an interval of length $2$ and therefore it contains 
points of $\tilde{E}$. Suppose now that $r< 4$. Then, $\tilde{E}\cap B(x,r/8) \subset E$ and $E\setminus \overline{B}(x,r/8)\neq\emptyset$. If $E \cap \overline{B}(x,r/8) = \{x\}$ then $\tilde{E} \cap \overline{B}(x,r/8) = \{x\}$ 
and the relative connectedness is satisfied with $M'=8$. If $E \cap \overline{B}(x,r/8) \neq \{x\}$ then, by the relative connectedness of $E$, $E\cap(\overline{B}(x,r)\setminus B(x,r/(8M))) \neq \emptyset$.

We show now the second claim. Recall that by Theorem \ref{thm:TroVa} $f$ is power quasisymmetric. Let $y,x,z \in \tilde{E}$ and assume $y\in \tilde{E}_{n_1}$, $x\in \tilde{E}_{n_2}$ and $z\in \tilde{E}_{n_3}$ with 
$n_1, n_2, n_3 \in \Z$. If $n_1=n_2=n_3$ the claim follows trivially. If $n_1, n_2, n_3$ are all different then
\[\frac{|\tilde{f}(x)-\tilde{f}(y)|}{|\tilde{f}(x)-\tilde{f}(z)|} \leq \frac{2|n_2-n_1| +1}{2|n_3-n_2|-1} \leq 9\frac{|n_2-n_1|}{|n_3-n_2| + 2} \leq 9\frac{|x-y|}{|x-z|}.\]
If $n_1=n_2 \neq n_3$ then the second inequality in Lemma \ref{lem:Holder} gives
\[\frac{|\tilde{f}(x)-\tilde{f}(y)|}{|\tilde{f}(x)-\tilde{f}(z)|} \leq C'\frac{|x-y|^{\a}}{|n_3-n_2|} \leq 3C'\left(\frac{|x-y|}{|x-z|}\right )^{\a}.\]
The remaining case $n_1\neq n_2 = n_3$ is treated similarly using the first inequality of Lemma \ref{lem:Holder}.
\end{proof}

By Lemma \ref{lem:unbound} we may assume for the rest that $E$ is a relatively connected closed set with no upper or lower bound. Hence, all components of $\R\setminus E$ are bounded open intervals.

For the second extension, we treat the case when $E$ has isolated points. For each isolated point $x\in E$ let $\pi(x) \in E$ be the closest point of $E\setminus\{x\}$ to $x$ and define
\[ E_x = \overline{B}(x,|x-\pi(x)|/10)\]
and $f_x : E_x \to \R^n$ with
\[ f_x(y) = f(x) + \frac{1}{\eta(1)}\frac{|f(x)-f(\pi(x))|}{|x-\pi(x)|}(y-x)\be_1.\]
If $x$ is an accumulation point of $E$, then set $E_x = \{x\}$ and $f_x : \{x\} \to \R$ with $f_x(x) = f(x)$. Finally, set $\hat{E} = \bigcup_{x\in E}E_x$ and $\hat{f} \colon \hat{E} \to \R$ with $\hat{f}|E_x = f_x$. Similar 
extensions also appear in a paper of Semmes \cite[Section 2]{Semmes3}.

\begin{remark}\label{rem:ratios}
Suppose that $x\in E$ is an isolated point. Then,
\[4 \leq d^*(E_x, \hat{E}\setminus E_x) \leq 5 \quad\text{ and } \quad 3 \leq d^*(\hat{f}(E_x), \hat{f}(\hat{E}\setminus E_x)) \leq 5\eta(1).\]
\end{remark}

The first claim of Remark \ref{rem:ratios} is clear. For the upper bound of the second claim note that $\dist(\hat{f}(E_x),\hat{f}(\hat{E}\setminus E_x)) \leq |f(x)-f(\pi(x))|\leq 5\eta(1)\diam{\hat{f}(E_x)}$. For the lower bound, 
take points $x'\in E_x$ and $y'\in \hat{E} \setminus E_x$ and assume that $y'\in E_{y}$. Then, 
%$|\hat{f}(x') - \hat{f}(y')| \geq |\hat{f}(x) - \hat{f}(y)| - |\hat{f}(x) - \hat{f}(x')|-|\hat{f}(y) - \hat{f}(y')| \geq \frac45|\hat{f}(x) - \hat{f}(y)|$ and

%$d^*(\hat{f}(E_x),\hat{f}(\hat{E}\setminus E_x)) \leq |f(x)-f(\pi(x))|(|f(x)-|)$
%\begin{align*}
%|\hat{f}(x)-\hat{f}(x')| &\leq \frac{|\hat{f}(x)-\hat{f}(\pi(x))|}{10\eta(1)} \leq \frac{1}{10\eta(1)}\eta\left( \frac{|x-\pi(x)|}{|x-y|}\right )|\hat{f}(x)-\hat{f}(y)|\\ 
%&\leq \frac{1}{10\eta(1)}|\hat{f}(x)-\hat{f}(y)|.
%\end{align*}
\begin{equation}\label{eq:ratios}
\frac{|\hat{f}(x')-\hat{f}(x)|}{|\hat{f}(x')-\hat{f}(y')|} 
\leq \frac{1}{10\eta(1)}\eta\left( \frac{|x-\pi(x)|}{|x-y|}\right )\frac{|\hat{f}(x)-\hat{f}(y)|}{\frac45 |\hat{f}(x)-\hat{f}(y)|} 
\leq \frac{1}{8}.
\end{equation}
Thus, if $x'$ is an endpoint of $E_x$, (\ref{eq:ratios}) yields $\dist(\hat{f}(x'),\hat{f}(\hat{E}\setminus E_x))\geq 4\diam{\hat{f}(E_x)}$. Hence, $\dist(\hat{f}(E_x),\hat{f}(\hat{E}\setminus E_x))\geq 3\diam{\hat{f}(E_x)}$ and the 
lower bound follows.

\begin{lem}\label{lem:isolated}
The set $\hat{E}$ is closed and $c$-uniformly perfect and $\hat{f}: \hat{E}\to\R^n$ is $\hat{\eta}$-quasisymmetric where $c$ depends only on $M$ and $\hat{\eta}$ depends only on $\eta$.
\end{lem}

\begin{proof}
Clearly, $E_x\cap E_y = \emptyset$ for $x,y\in E$ with $x\neq y$. To see that $\hat{E}$ is closed, take $y \in \overline{\hat{E}}$. If $y\in \overline{\hat{E}} \setminus E$ then $y \in \overline{E_x}$ for some $x\in E$ and, thus, 
$y\in \hat{E}$.

Since $\hat{E}$ has no isolated points, we only need to show that $\hat{E}$ is $M'$-relatively connected for some $M'$ depending on $M$. Take $x\in \hat{E}$ and $r>0$. From the unboundedness of $\hat{E}$ and the fact that $\hat{E}$ 
has no isolated points, we have $\{x\} \subsetneq \overline{B}(x,r)\cap \hat{E} \subsetneq \hat{E}$. If $x\in E$ is not isolated in $E$, then 
\[\emptyset \neq E\cap(\overline{B}(x,r)\setminus B(x,r/M)) \subset \hat{E}\cap(\overline{B}(x,r)\setminus B(x,r/M)).\]
Suppose $x \in E_z$ for some isolated point $z$ in $E$. If $r > 2M\dist(z,E\setminus\{z\})$ then $\emptyset \neq (E\setminus\{z\})\cap \overline{B}(z,r/2) \subset \hat{E}\cap \overline{B}(x,r)$. Therefore,
\[\emptyset \neq E\cap(\overline{B}(z,r/2)\setminus B(z,(2M)^{-1}r)) \subset \hat{E}\cap(\overline{B}(x,r)\setminus B(x,(4M)^{-1}r)).\] 
If $r \leq 2M\dist(z,E\setminus\{z\})$ then $(20M)^{-1}r \leq \frac{1}{10}\dist(z,E\setminus\{z\})$ and
\[\emptyset \neq E_z\cap(\overline{B}(x,r)\setminus B(x,(20M)^{-1}r)) \subset \hat{E}\cap(\overline{B}(x,r)\setminus B(x,(20M)^{-1}r)).\]

It remains to show that $\hat{f}$ is quasisymmetric; then by Theorem \ref{thm:TroVa} $\hat{f}$ will be power quasisymmetric. Let $x,y,z \in \hat{E}$ be three distinct points with $x\in E_{x'}$, $y\in E_{y'}$ and $z\in E_{z'}$ for 
some $x',y',z' \in E$. If $x'=y'=z'$ then $x,y,z$ are in an interval where $\hat{f}$ is a similarity.

If $x'\neq z'$ and $x' = y'$ then, by Remark \ref{rem:ratios}, the prerequisites of Lemma 2.29 in \cite{Semmes3} are satisfied for $A = E\setminus \{x'\}$, $A^* = E \cup E_{x'}$ and $H=\hat{f}|A^*$ and $\hat{f}|E\cup E_{x'}$ is 
$\eta'$-quasisymmetric for some $\eta'$ depending only on $\eta$. Hence,
\[ \frac{|\hat{f}(x)-\hat{f}(y)|}{|\hat{f}(x)-\hat{f}(z)|} \leq C_1\frac{|\hat{f}(x)-\hat{f}(y)|}{|\hat{f}(x)-\hat{f}(z')|} \leq C_1\eta'\left (\frac{|x-y|}{|x-z'|} \right ) \leq C_1 \eta'\left (C_2\frac{|x-y|}{|x-z|} \right )\]
for some $C_1,C_2>1$ depending only on $\eta$. Similarly for $x' = z' \neq y'$. If $x',y',z'$ are distinct then by Remark \ref{rem:ratios},
\begin{align*}
\frac{|\hat{f}(x)-\hat{f}(y)|}{|\hat{f}(x)-\hat{f}(z)|} &\leq C_3\frac{|\hat{f}(x')-\hat{f}(y')|}{|\hat{f}(x')-\hat{f}(z')|} \leq C_3 \eta\left (\frac{|x'-y'|}{|x'-z'|}\right ) \leq  C_3 \eta\left (C_4\frac{|x-y|}{|x-z|}\right )
\end{align*}
for some constants $C_3,C_4 >1$ depending only on $\eta$. Thus, $\hat{f}$ is quasisymmetric.
\end{proof}

\subsection{Bridges}\label{sec:bridges}
By Lemma \ref{lem:unbound} and Lemma \ref{lem:isolated}, we may assume that $E$ is a closed $c$-uniformly perfect set such that every component of $\R \setminus E$ is a bounded open interval, and $f: E \to \R^n$ is an 
$\eta$-quasisymmetric embedding.

In this section, for each component $I$ of $\R\setminus E$, we construct a path in a higher dimensional space $\R^{N}$, $N\geq n$, connecting the images of the endpoints of $I$. The union of these paths along with $f(E)$ gives a 
homeomorphic extension $F:\R \to \R^{N}$.
%of $f$.

For two points $x,y \in \R^n \subset \R^k$ let $T_k(x,y)$ be the equilateral triangle which contains the line segment $[x,y]$ and lies on the $2$-dimensional plane defined by the points $x$, $y$ and $\be_{k}$. The \emph{bridge} of 
$x$ and $y$ in dimension $k$, denoted by $\cB_k(x,y)$, is the closure of $T_{k}(a,b)\setminus [x,y]$. 

\begin{remark}\label{rem:distbridges}
If $z,a,b\in \R^n$ with $|z-a| \leq |z-b|$ then, for all $x\in \cB_k(a,b)$, $|z-x| \geq C^{-1}(|z-a| + |x-a|)$ for some universal $C>1$.
\end{remark}

\begin{remark}\label{rem:bridgeseq}
Each bridge $\cB_k(x,y)$ is $4$-bi-Lipschitz equivalent to a closed interval of $\R$ of length $|x-y|$.
\end{remark}

Using Remark \ref{rem:distbridges} and triangle inequality, it is easy to verify that the relative distance of two bridges $\cB_{k}(x_1,y_1)$ and $\cB_{m}(x_2,y_2)$, with $k\neq m$, is comparable to the relative distance of the sets 
$\{x_1,y_1\}$ and $\{x_2,y_2\}$.

\begin{remark}\label{rem:relbridges}
Let $n,m_1,m_2 \in \N$ with $n< m_1 \leq m_2$ and let $x_1,y_1,x_2,y_2 \in \R^n$. There exists a universal $C_1>0$ such that 
\[ d^*(\cB_{m_2}(x_1,y_1),\cB_{m_1}(x_2,y_2)) \leq C_1 d^*(\{x_1,y_1\},\{x_2,y_2\}).\]
On the other hand, there exist universal constants $d_0>0$ and $C_2>0$ such that $d^*(\{x_1,y_1\},\{x_2,y_2\}) \geq d_0$ implies 
\[ d^*(\{x_1,y_1\},\{x_2,y_2\}) \leq C_2 d^*(\cB_{m_2}(x_1,y_1),\cB_{m_1}(x_2,y_2)). \]
\end{remark}

For each component $I$ of $\R\setminus E$ we denote by $a_I,b_I$ the endpoints of $I$ with $a_I<b_I$ and by $m_I$ the center of $I$. We also write $\cB_k(I) = \cB_k(f(a_I),f(b_I))$ where $k>n$. In general, two bridges $\cB_k(I)$ and 
$\cB_k(I')$, with $I\neq I'$, may intersect. Therefore, more dimensions are needed to make sure that such an intersection will never happen. The next lemma allows us to use only a finite amount of dimensions for this purpose.

\begin{lem}\label{lem:dim}
Let $d>0$. If $I_1,\dots,I_k$ are mutually disjoint closed intervals in $\R$ with $d^*(I_i,I_j) \leq d$ for all $i,j = 1,\dots,k$, $i\neq j$, then $k \leq 2d+3$.
\end{lem}

\begin{proof}
%Let $I_1,\dots,I_k$ be mutually disjoint intervals in $\R$ with $k\geq 3$ and $d^*(I_i,I_j) \leq d$ for all $i\neq j$. 
We may assume that if $i\not\in \{1,k\}$, $x\in I_1$, $y\in I_i$ and $z\in I_k$ then $x<y<z$. Furthermore, 
%$I_1$, $I_k$ is bounded; suppose that $I_1$ is bounded and $\diam{I_1} < \diam{I_k}$. 
applying a similarity we may assume that $\dist(I_1,I_k) = 1$.
 
Since $d^*(I_1,I_k) \leq d$, we have $\diam{I_1}\wedge\diam{I_k} \geq d^{-1}$. Since the intervals $I_2,\dots,I_{k-1}$ are between $I_1$ and $I_k$, there exists at least one $j\in\{2,\dots,k-1\}$ such that 
$\diam{I_j} \leq \dist(I_1,I_k)/(k-2) = (k-2)^{-1}$. Thus, $\dist(I_1,I_j)\vee\dist(I_k,I_j) \geq \frac12(1-\frac{1}{k-2})$. If $\diam{I_j}\geq d^{-1}$ then $k\leq d+2$. Otherwise,
\[ d \geq d^*(I_1,I_j)\vee d^*(I_k,I_j) 
  \geq \frac{\dist(I_1,I_j)\vee\dist(I_k,I_j)}{d^{-1}\wedge\diam{I_j}}
  %\geq \frac{\frac12(1-\frac{d}{k-2})}{2d(1\wedge \frac{d}{k-2})}
  \geq \frac{1}{2}(k-3).\qedhere\]
%Therefore, $k\leq 4d^3 +d+4$.
\end{proof}

Let now $I_1,I_2,\dots$ be an enumeration of the components of $\R\setminus E$. 
%By Remark \ref{rem:rel}, there exists $C_1>0$ depending only on $c$ such that $d^*(\overline{I_i},\overline{I_j}) > C_1$ for all $i\neq j$. 
By Remark \ref{rem:relbridges} and (\ref{eq:relQS}), there exists $C_0>0$ so that $d^*(\overline{I_i},\overline{I_j}) \geq C_0$ implies $d^*(\cB_{m}(I_i),\cB_{m}(I_j)) \geq 1$ for all $m>n$. By Lemma \ref{lem:dim}, there exists 
$n_0\in \N$, depending only on $c$ and $\eta$, such that if distinct $J_1,\dots,J_k \in \{I_1,I_2,\dots\}$ with $d^*(J_{i},J_j) < C_0$ for all $i\neq j$ then $k\leq n_0$. 
Set $N = n+ n_0+1$. Let $\cB_{n_{I_1}}(I_1)$ be the bridge with $n_{I_1} = n+1$. Suppose that $\cB_{n_{I_1}}(I_1), \dots, \cB_{n_{I_m}}(I_m)$ have been defined. Then, there exist at most $n_0$ indices $i_1,\dots,i_k$ in 
$\{1,\dots,m\}$ such that $d^*(I_{m+1},I_{i_j}) < C_0$. Pick $n_{I_{m+1}} \in \{n+1,\dots,N\}\setminus \{n_{I_{i_1}},\dots,n_{I_{i_{k}}}\}$ and define the bridge $\cB_{n_{I_{m+1}}}(I_{m+1})$. Inductively, for each component $I$ of 
$\R\setminus E$ we obtain a bridge $\cB_{n_I}(I)$ with $n_I \leq N$.

\begin{cor}\label{cor:bridgeimages}
Set $I' = \{f(a_I),f(b_I)\}$ for any component $I = (a_I,b_I)$ of $\R\setminus E$. Then, there exist $C>1$ depending only on $c$ and $\eta$ such that, for every two components $I,J$ of $\R\setminus E$ with $I\neq J$,
\[ (C)^{-1} d^*(I',J') \leq d^*(\cB_{n_I}(I), \cB_{n_J}(J)) \leq C d^*(I',J') \]
and $C^{-1}\dist(I',J') \leq \dist(\cB_{n_I}(I), \cB_{n_J}(J)) \leq C \dist(I',J')$.
\end{cor}

\subsection{Reflected sets and functions}\label{sec:reflect}
As before, we assume that $E$ is a closed $c$-uniformly perfect set such that every component of $\R \setminus E$ is a bounded open interval, and $f: E \to \R^N$ is an $\eta$-quasisymmetric embedding with $N=n+n_0+1$.

Recall from Section \ref{sec:bridges} that, given a a component $I = (a_I,b_I)$ of $\R\setminus E$, we denote by $m_{I}$ the midpoint of $I$. Moreover, we denote by $m_{\cB(I)}$ the point in $\cB_{n_I}(I)$ such that 
$\cB_{n_I}(I) = [f(a_I),m_{\cB(I)}] \cup [f(b_I),m_{\cB(I)}]$. Note that $[f(a_I),m_{\cB(I)}] \cap [f(b_I),m_{\cB(I)}] = \{m_{\cB(I)}\}$.

Let $I = (a_I,b_I)$ be a component of $\R\setminus E$. We define an increasing sequence in $E$ converging to $a_I$ as follows. Set $\d_0 = \min\{1/2,\eta^{-1}(1/2)\}$. Since $E$ is uniformly perfect, there exists $a_0\in E$, 
$a_0 < a_I$ with $|a_0 -a_I| \in [(2c)^{-1}|I|,2^{-1}|I|]$. Inductively, suppose that $a_k$ has been defined. Since $E$ is uniformly perfect, there exists $a_{k+1}\in E \cap (a_k,a_I)$ such that
\[ \frac{\d_0}{c} \leq \frac{|a_{k+1}-a_I|}{|a_k-a_I|} \leq \d_0. \]
Let $a'_0 = m_I$ and for each $k\geq 1$ let $a'_k \in (a_I,m_I)$ with $a'_k = 2a_I - a_k$. Similarly we obtain sequences $\{b_k\}_{k\geq 0} \subset E$ and $\{b_k'\}_{k\geq 0}\subset [m_I,b_I]$ for the point $b_I$. In the following, 
two intervals $[a_{k+1}',a_k']$ and $[a_{k}',a_{k-1}']$ are called \emph{neighbor intervals}. Similarly, $[a_1',m_{I}]$ is a neighbor of $[m_{I},b_1']$ and for each $k\in\N$, $[b_{k-1}',b_k']$ is a neighbor of $[b_{k}',b_{k+1}']$.

We define now $f_I: \overline{I} \to \cB_{n_I}(I)$. Set $f_I(m_I) = m_{\cB(I)}$ and for each $k\geq 1$, define $f_I(a_k')\in[f(a_I),m_{\cB(I)}]$ and $f_I(b_k') \in[f(b_I),m_{\cB(I)}]$ by
\[ \frac{|f_I(a_k') - f(a_I)|}{|f(a_k) -f(a_I)|} = 1 = \frac{|f_I(b_k') - f(b_I)|}{|f(b_k) -f(b_I)|}.\]
On each interval $[a'_{k+1},a_k']$ or $[b_k',b'_{k+1}]$ we extend $f_I$ linearly. It follows from the choice of $\d_0$ that $f_I$ is a homeomorphism.

Suppose that $J_1, J_2 \subset I$ are neighbor intervals. Then, there exists constant $C>1$ depending only on $\eta$ and $c$ such that 
\begin{equation}\label{eq:const1}
C^{-1} \leq |J_1|/|J_2|<C \text{ and } C^{-1} \leq \diam{f_I(J_1)}/\diam{f_I(J_2)}<C.
\end{equation}
Thus, by Lemma \ref{lem:jointqs}, Remark \ref{rem:bridgeseq} and the linearity of $f_I$ on each $J_i$ the following remark can be easily verified.

\begin{remark}\label{rem:qs}
Suppose that $J_1,J_2,J_3 \subset I$ are consecutive neighbor intervals. Then, there exists $\eta_1$ depending only on $\eta$ and $c$ such that $f_I|(J_1\cup J_2 \cup J_3)$ is $\eta_1$-quasisymmetric.
\end{remark}

Note that $f_I|\{a_k'\}_{k\geq 0}$ is $\eta_2$-quasisymmetric for some $\eta_2$ depending only on $\eta$ and $c$. We show in the next lemma that $f_I$ is quasisymmetric.

\begin{lem}\label{lem:key}
Let $I$ be a component of $\R\setminus E$. There exists $\eta'$ depending only on $\eta$ and $c$ such that $f_I$ is $\eta'$-quasisymmetric.
\end{lem}

\begin{proof}
%By Lemma \ref{lem:ext} it suffices to show that $f_I$ is weakly quasisymmetric. 
By Remark \ref{rem:qs}, $f_I|[a_1',b_1']$ is quasisymmetric. We show that 
$f_I|[a_I,a_0']$ is quasisymmetric and similar arguments apply for $f_I|[b_0',b_I]$. Then, by Lemma \ref{lem:jointqs} and Remark \ref{rem:bridgeseq}, $f_I$ is $\eta'$-quasisymmetric with $\eta'$ depending only on $\eta$ and $c$.
Recall that $f_I|\{a_k'\}_{k\geq 0}$ is $\eta_2$-quasisymmetric with $\eta_2$ depending only on $\eta$ and $c$.

To show that $f_I|[a_I,a_0']$ is quasisymmetric, we apply Lemma \ref{lem:jointqs}. Let $x,y,z$ be in $[a_I,a_0']$, with $x$ being between $y$ and $z$, and $|x-y| \leq |x-z|$. Suppose $x\in[a_k',a_{k-1}']$. 

Assume first that $y<x<z$. If $z\geq a_{k-2}'$ then $|f_I(x)-f_I(y)| \leq |f_I(a_{k-1}')-f_I(a_I)|\leq \eta_2(2)|f_I(a_{k-1}')-f_I(a_{k-2}')| \leq \eta_2(2)|f_I(x)-f_I(z)|$. If $z\leq a_{k-2}'$ and $y\geq a_{k+1}'$ then the 
quasisymmetry follows from Remark \ref{rem:qs}. If $z\leq a_{k-2}'$ and $y\leq a_{k+1}'$ then $|x-z|\geq |x-y|\geq C^{-1}|a_{k-1}'-a_{k}'|$ and by Remark \ref{rem:qs},
$|f_I(x)-f_{I}(y)| \leq |f_I(a_{k-1}')-f_I(a_I)| \leq \eta_2(2)|f_I(a_k')-f_I(a_{k-1}')| \leq \eta_2(2)(|f_I(x)-f_I(a_k')| + |f_I(x)-f_I(a_{k-1}')|) \leq 2\eta_2(2)\eta_1(C)|f_I(x)-f_I(z)|$.

%are in the same interval $[a_n',a_{n-1}']$ or in neighbor intervals then $y$ has to be in the same interval as $x$ or in a neighbor interval of $x$. The quasisymmetry follows then from Remark 
%\ref{rem:qs}. If not, there exists maximal $n\in\N$ such that $x \leq a_{n+1}' < a_n' \leq z$. Then, 
%\[ \frac{|f_I(x)-f_I(y)|}{|f_I(x)-f_I(z)|} \leq \frac{|f_I(a_n')-f_I(a_I)|}{|f_I(a_{n+1}')-f_I(a_n')|} \leq \eta_2\left ( \frac{|a_n'- a_I|}{|a_n'-a_{n+1}'|} \right )  \leq  \eta_2(2).\]

Assume now that $z<x<y$. Then, there exists $m_0\in \N$ depending only on $c$ and $\eta$ such that $y \leq a_{k-m}'$ for some $0\leq m \leq m_0$. If $z\geq a_{k+1}'$ then we obtain quasisymmetry by applying Lemma \ref{lem:jointqs} at 
most $m_0$ times. If $z\leq a_{k+1}'$, then $|f_I(x)-f_I(y)| \leq |f_I(a_k')-f_I(a_{k-m}')| \leq \eta_2(m_0C^{m_0})|f_I(a_k')-f_I(a_{k+1}')| \leq \eta_2(m_0C^{m_0})|f_I(x)-f_I(z)|$ where $C$ is as in (\ref{eq:const1}).
\end{proof}
%\[ \frac{|f_I(x)-f_I(y)|}{|f_I(x)-f_I(z)|} \leq \frac{|f_I(a_n')-f_I(a')|}{|f_I(a_{n+1}')-f_I(a_n')|} \leq \eta_2\left ( \frac{|a_n'- a'|}{|a_n'-a_{n+1}'|} \right )  \leq  \eta_2(2C^2)\]
%where $a' = a_{n-1}'$ if $x$ and $y$ are in the same interval or $a' = a_{n-2}'$  if $x$ and $y$ are in neighbor intervals.

\section{Proof of main results}\label{sec:proof}

We show Theorem \ref{thm:main2} in this section. The proof of Theorem \ref{thm:main} is given in Section \ref{sec:proof2} and is a minor modification of that of Theorem \ref{thm:main2}.

Let $N = n+n_0+1$ be as in Section \ref{sec:bridges}. Define $F: \R \to \R^N$ with $F|E = f$ and $F|I = f_I$ whenever $I$ is a component of $\R\setminus E$. 
%Clearly $F$ is an embedding. 
We show in Section \ref{sec:proof1} that $F$ satisfies (\ref{eq:weakQS}) and then, Lemma \ref{lem:ext} concludes the proof of Theorem \ref{thm:main2}.

To limit the use of constants we write in the following $u\lesssim v$ (resp. $u \simeq v$) when the ratio $u/v$ is bounded above (resp. bounded above and below) by a positive constant depending at most on $\eta$ and $c$.

\subsection{A form of monotonicity}

For the proof of the quasisymmetry of $F$ we show first that $F$ satisfies the following form of monotonicity.

\begin{lem}\label{lem:mon}
Suppose that $x_1,x_2,x_3 \in \R$ with $x_1 < x_2<x_3$. Then, 
\[ |F(x_2)-F(x_1)|\vee|F(x_3)-F(x_2)| \lesssim |F(x_3)-F(x_1)|. \]
\end{lem}

First we make an observation.
%Let $x\in I$ and $y\in E$. Denote by $y'$ the pont of $\overline{I}$ closest to $y$. Then, by Remark \ref{rem:distbridges} and the quasisymmetry of $f$, we have that \begin{align}\label{eq:dist3}|F(y)-F(x)|  &\simeq 
%|F(y) - F(y')| + |F(y') - F(x)|. \end{align}
Let $x,y\in\R$ with $x<y$ that are not on the same component of $\R\setminus E$. Denote by $x',y'$ the minimum and maximum, respectively, of $E\cap[x,y]$. By Corollary \ref{cor:bridgeimages} and the quasisymmetry of $f$,
\begin{equation}\label{eq:dist4}
|F(x)-F(y)| \simeq |F(x)-F(x')| + |F(x')-F(y')| + |F(y') - F(y)|.
\end{equation}

\begin{proof}[{Proof of Lemma \ref{lem:mon}}]
Let $x_1,x_2,x_3 \in \R$ with $x_1 < x_2 < x_3$. We only show that $|F(x_2)-F(x_1)|\lesssim |F(x_3)-F(x_1)|$; the inequality $|F(x_2)-F(x_1)|\lesssim |F(x_3)-F(x_1)|$ is similar. 

If all three of them are in $E$ or in the same component $I$ of $\R\setminus E$ then the claim follows from the quasisymmetry of $f$ and $f_I$. Therefore, we may assume that at least one of the $x_1,x_2,x_3$ is in $\R\setminus E$.

\emph{Case 1}. Suppose that there exists a component $I$ of $\R \setminus E$ that contains exactly two of the $x_1,x_2,x_3$. Assume, for instance that $x_1,x_2\in I$ and $x_3\not\in I$; the case $x_2,x_3 \in I$ is similar. Let 
$x_2'$ and $x_3'$ be the minimum and maximum, respectively, of $E\cap [x_2,x_3]$. By (\ref{eq:dist4}) and the quasisymmetry of $F$ on $I$, $|F(x_3)-F(x_1)| \gtrsim |F(x_2')-F(x_1)| \gtrsim |F(x_2)-F(x_1)|$. 

\emph{Case 2}. Suppose that there is no component of $\R\setminus E$ containing two points from $x_1,x_2,x_3$. Let $x_1'$ and $x_2'$ be the minimum and maximum, respectively, of $E \cap [x_1,x_2]$ and  $x_2'',x_3'$ be the minimum and 
maximum, respectively, of $E\cap[x_2,x_3]$. Applying (\ref{eq:dist4}) on $x_1,x_3$ and quasisymmetry on $x_1',x_2'',x_3'$, 
$|F(x_3)-F(x_1)| \gtrsim |F(x_2'')-F(x_2')| + |F(x_2')-F(x_1')| + |F(x_1')-F(x_1)|$. Applying quasisymmetry on $x_2',x_2,x_2''$ and then (\ref{eq:dist4}) on $x_1,x_2$, 
$|F(x_3)-F(x_1)| \gtrsim |F(x_2)-F(x_2')| + |F(x_2')-F(x_1')| + |F(x_1')-F(x_1)| \gtrsim |F(x_2)-F(x_1)|$.
%By (\ref{eq:dist4}), $|F(x_1)-F(x_i)| \simeq |F(x_1)-F(x_1')| + |F(x_1')-F(x_i)|$ for $i=2,3$ and we may assume $x_1=x_1'\in E$. We may also assume that $x_3=x_3'\in E$. Applying 
%quasisymmetry of $f$ on points $x_1 \leq x_2' \leq x_3$ and on points $x_1\leq x_2''\leq x_3$,
%\begin{align*}
%|F(x_3) - F(x_1)| &\gtrsim |F(x_1) - F(x_2')|\vee |F(x_1) - F(x_2'')| \gtrsim |F(x_1)-F(x_2)|
%\qedhere
%\end{align*}
\end{proof}

\subsection{Proof of Theorem \ref{thm:main2}}\label{sec:proof1}

Let $x,y,z \in \R$ such that $|x-y|\leq |x-z|$. By Lemma \ref{lem:mon}, we may assume that $x$ is between $y$ and $z$. Without loss of generality we assume that $y<x<z$. 

%If $x,y,z \in E$ then $|F(x)-F(y)|\lesssim |F(x)-F(z)|$ by the quasisymmetry of $f$. 
Since $F|E$ is already quasisymmetric, we may assume that at least one of the $x,y,z$ is in $\R\setminus E$. The proof is divided in four cases.

For the first case, we use the following lemma that can easily be verified.

\begin{lem}\label{lem:replace}
Let $I = (a,b)$ be a component of $\R\setminus E$, $x_1\in I$ and $x_2\in E$. 

Suppose $x_1<x_2$. If $|x_2-b| > (4c)^{-1}|x_1-b|$ set $x_1'=b$. If $|x_2-b| \leq (4c)^{-1}|x_1-b|$ and $x_1\leq m_I$ set $x_1' = b_0$. If $|x_2-b| \leq (4c)^{-1}|x_1-b|$ and $x_1\in [b'_{n+1},b'_n]$ set $x_1' = b_{n+1}$. In each 
case, $|x_2-x_1'| \simeq |x_2-x_1|$ and  $|F(x_2)-F(x_1')| \simeq |F(x_2)-F(x_1)|$.

If $x_2<x_1$ replace $b,b_0,b_{n+1}$ by $a,a_0,a_{n}$, respectively, and define $x_1'$ similarly. The claim of the lemma holds in this case as well.
\end{lem}

\emph{Case 1}. Suppose that exactly one of the $x,y,z$ is in $\R\setminus E$.

\emph{Case 1.1}. Assume that $y \in \R\setminus E$ and $x,z\in E$. Let $y'$ be as in Lemma \ref{lem:replace} for the pair $x_1=y$, $x_2=x$. Then, $|y'-x|\simeq |y-x| \lesssim |x-z|$ and 
\[|F(y)-F(x)| \simeq |F(y')-F(x)| \lesssim |F(x)-F(z)|.\]

\emph{Case 1.2}. Assume that $z\in \R\setminus E$ and $x,y\in E$. We work as in Case 1.1.

\emph{Case 1.3}. Assume that $x \in \R\setminus E$ and $y,z\in E$. Let $x'$ be the point defined in Lemma \ref{lem:replace} for the pair $x_1=x$, $x_2=z$. Then, $|y-x'| = |y-x| + |x-x'| \lesssim |x-z| \simeq |x'-z|$ and by Lemma 
\ref{lem:mon},
\[ |F(x)-F(y)| \lesssim |F(x')-F(y)| \lesssim |F(x')-F(z)| \simeq |F(x)-F(z)|.\]

\emph{Case 2}. Suppose that exactly two of the $x,y,z$ are in the same component of $\R\setminus E$ and the third point is in $E$.

\emph{Case 2.1}. Assume that $x,y$ are in a component $(a,b)$ of $\R\setminus E$ and $z\in E$. 

If $|x-b| > |b-z|$ set $z' = b$. Note that $|x-z| \simeq |x-z'|$ and, by quasisymmetry of $F|(a,b)$ and Lemma \ref{lem:mon},
\[ |F(x)-F(y)| \lesssim |F(x)-F(z')| \lesssim |F(x)-F(z)|.\]

If $|x-b| \leq |b-z|$ then set $x'=b$. Note that $|x-y| \leq |x'-y|\lesssim |x-z| \simeq |x'-z|$. By Lemma \ref{lem:mon} and Case 1 for  $y,x',z$,
\[ |F(x)-F(y)| \lesssim |F(x')-F(y)| \lesssim |F(x')-F(z)| \lesssim |F(x)-F(z)|.\]
%\[ \frac{|F(x)-F(y)|}{|F(x)-F(z)|} \lesssim  \frac{|F(x')-F(y)|}{|F(x')-F(z)|}.\]
%For the points $y,x',z$ apply now Case 1.

\emph{Case 2.2}. Assume that $x,z$ are in a component $(a,b)$ of $\R\setminus E$ and $y\in E$. 
If $|y-a|\leq|x-a|$ set $y'=a$ and if $|y-a|>|x-a|$ then set $x'=a$. In each case we work as in Case 2.1. 
%Note that $|y-x|\simeq |y'-x|$ and, by Lemma \ref{lem:key},\[ |F(x)-F(y)| \lesssim |F(x)-F(y')| \lesssim |F(x)-F(z)|.\]
%If $|y-a|>|x-a|$ then set $x'=a$ and work as in Case 2.1(2). 
%Note that $|y-x'|\simeq |x-y| \leq |x-z| \simeq |x'-z|$, $|F(x)-F(y)| \simeq |F(x')-F(y)|$ and $|F(x)-F(z)| \simeq |F(x')-F(z)|$. Then, apply Case 1 for the points $y,x',z$.

For the next two cases we use the following lemma.

\begin{lem}\label{lem:replace2}
Let $(a_1,b_1)$, $(a_2,b_2)$ be two components of $\R\setminus E$ with $b_1<a_2$ and $x_1\in (a_1,b_1)$, $x_2\in (a_2,b_2)$. 

If $|a_1-b_1| \leq |a_2-b_2|$ set $x_1' = b_1$. Then, $|x_1-x_2| \simeq |x_1'-x_2|$ and $|F(x_2) - F(x_1)| \simeq |F(x_2) - F(x_1')|$.

If $|a_1-b_1| > |a_2-b_2|$ set $x_2' = a_2$. Then, $|x_1-x_2| \simeq |x_1-x_2'|$ and $|F(x_2) - F(x_1)| \simeq |F(x_2') - F(x_1)|$.
\end{lem}

\begin{proof}
Assume that $|a_1-b_1| \leq |a_2-b_2|$; the case $|a_2-b_2| \leq |a_1-b_1|$ is similar. By Remark \ref{rem:rel}, $|x_1-x_2| \simeq |x_1'-x_2|$. Moreover, by Lemma \ref{lem:mon},
\begin{align*}
|F(x_2) - F(x_1')| &\lesssim |F(x_2) - F(x_1)| \leq |F(x_2)-F(x_1')| + |F(x_1')-F(a_1)| \\
&\lesssim |F(x_2)-F(x_1')| + |F(x_1')-F(a_2)| \lesssim |F(x_2) - F(x_1')|.\qedhere
\end{align*}
\end{proof}

\emph{Case 3}. Suppose that exactly two of the $x,y,z$ are in $\R\setminus E$ but in different components.

\emph{Case 3.1}. Assume that $y\in(a_1,b_1)$, $x\in(a_2,b_2)$ and $z\in E$ where for each $i=1,2$, $(a_i,b_i)$ is a component of $\R\setminus E$ and $b_1 < a_2$. 

If $|a_1-b_1| \leq |a_2-b_2|$ then, by Lemma \ref{lem:replace2}, setting $y'=b_1$, we have $|x-y'| \simeq |x-y|$, $|F(x)-F(y')|\simeq |F(x)-F(y)|$. Now apply Case 1 for the points $y',x,z$.

If $|a_2-b_2| < |a_1-b_1|$ then, by Lemma \ref{lem:replace2}, setting $x' = a_2$, we have $|x'-y| \simeq |x-y|$ and $|F(x')-F(y)|\simeq |F(x)-F(y)|$. Moreover, $|x-z| \leq |x'-z| = |x-z| + |x-x'| \leq |x-z| +  |x-y| \leq 2|x-z|$. 
Thus, $|x-z|\simeq |x'-z|$ and applying Case 1 for the points $x',x,z$, we have $|F(z)-F(x)| \simeq |F(z)-F(x')|$. Now apply Case 1 for $y,x',z$.

\emph{Case 3.2}. Assume that $x\in(a_1,b_1)$, $z\in(a_2,b_2)$ and $y\in E$ where for each $i=1,2$, $(a_i,b_i)$ is a component of $\R\setminus E$ and $b_1 < a_2$. 

If $|a_1-b_1| \leq |a_2-b_2|$ then, $|x'-z|\simeq |x-z|$, $|F(x')-F(z)| \simeq |F(x)-F(z)|$, $|y-x| \lesssim |y-x'| \lesssim |x'-z|$, $|F(y)-F(x)| \lesssim |F(y)-F(x')| \lesssim |F(x')-F(z)|$ and we apply Case 1 for $y,x',z$.

If $|a_2-b_2| < |a_1-b_1|$ then set $z' = a_2$ and work as in Case 3.1.

\emph{Case 3.3}. Assume that $y\in(a_1,b_1)$, $z\in(a_2,b_2)$ and $x\in E$ where for each $i=1,2$, $(a_i,b_i)$ is a component of $\R\setminus E$ and $b_1 < a_2$. 

If $|a_1-b_1| \leq |a_2-b_2|$ then set $y' = a_1$. Since $|x-z|\simeq|x-y|+|x-z| \gtrsim |b_1-a_2|$ we have that $|x-y'|\simeq |x-y|$. Moreover, by Lemma \ref{lem:mon}, $|F(x)-F(y)|\lesssim|F(x)-F(y')|$ and we apply Case 1 for 
$y',x,z$.

If $|a_2-b_2| < |a_1-b_1|$ then set $z' = b_2$. As before, $|x-z|\simeq |x-z'|$. Furthermore, $|F(x)-F(z')| \simeq |F(x)-F(a_2)|$ when $|x-a_2|>|a_2-z|$ and $|F(x)-F(z')| \simeq |F(b_2)-F(a_2)|$ when $|x-a_2|\leq|a_2-z|$. In either 
case, $|F(x)-F(z)|\simeq |F(x)-F(z')|$ and we apply Case 1 for the points $y,x,z'$.

\emph{Case 4}. Suppose that $y,x,z\in\R\setminus E$. By Lemma \ref{lem:key}, we may assume that either $y$ or $z$ is not in the same component as $x$.

%If all three of them are in the same component of $\R\setminus E$ then Lemma \ref{lem:key} applies. Thus we may assume that either $y$ and $x$ are in different components or $x$ and $z$ are in different components.

\emph{Case 4.1}. Assume that $y\in(a_1,b_1)$ and $x\in (a_2,b_2)$ where $(a_i,b_i)$ are components of $\R\setminus E$ and $b_1 < a_2$.

If $|b_1-a_1| \leq |b_2-a_2|$ then set $y' = b_1$ and, by Lemma \ref{lem:replace2}, $|x-y| \simeq |x-y'|$ and $|F(x)-F(y)|\simeq |F(x)-F(y')|$. Apply now Case 2 or Case 3 for the points $y',x,z$.

If $|b_2-a_2| < |b_1-a_1|$ then set $x' = a_2$ and, by Lemma \ref{lem:replace2}, $|x-y| \simeq |x'-y|$ and $|F(x)-F(y)| \simeq |F(x')-F(y)|$. As in Case 3.1, $|x-z|\simeq |x'-z|$ and applying Case 2 or Case 3 for the points $x',x,z$ 
we conclude that $|F(x)-F(x')| \lesssim |F(x)-F(z)|$ which implies $|F(x)-F(z)| \simeq |F(x')-F(z)|$. Now apply Case 2 or Case 3 on the points $y,x',z$.

\emph{Case 4.2}. Assume that $x\in(a_1,b_1)$, $z\in (a_2,b_2)$ where $(a_i,b_i)$ are components of $\R\setminus E$ and $b_1 < a_2$.

If $|b_2-a_2| \leq |b_1-a_1|$ then set $z' = a_2$ and work as in Case 4.1.
%, by Remark \ref{rem:replace2}, $|x-y| \simeq |x-z'|$ and $|F(x)-F(z)|\simeq |F(x)-F(z')|$. Apply now Case 2 or Case 3 for the points $y,x,z'$.

If $|b_1-a_1| < |b_2-a_2|$ then set $x' = b_1$ and, by Lemma \ref{lem:mon} and Lemma \ref{lem:replace2}, $|x'-y| = |x-y| + |x-x'| \lesssim |x-z|\simeq |x'-z|$, $|F(x')-F(z)| \simeq |F(x)-F(z)|$ and 
$|F(x)-F(y)| \lesssim |F(x')-F(y)|$. Apply now Case 2 or Case 3 for the points $y,x',z$.

\subsection{Proof of Theorem \ref{thm:main}}\label{sec:proof2}

By Corollary \ref{cor:relconnec} we only need to show the sufficiency in Theorem \ref{thm:main}. The proof is a mild modification of the proof of Theorem \ref{thm:main2}. We only outline the steps. 

Let $E\subset\R$ be an $M$-relatively connected set and let $f:E \to \R$ be a monotone $\eta$-quasisymmetric mapping. As before, we may assume that $E$ is a closed set that contains at least two points and $f$ is power 
quasisymmetric. Moreover, we may assume that $f$ is increasing.
 
\emph{Step 1}. First, we reduce the proof to the case that $E$ has no lower or upper bound, as in Section \ref{sec:prelim}. This time, however, in Case 1 and Case 2 we define $\tilde{f}(-a_n) = -a_n$, where $\{a_n\} \subset E$ is as 
in Section \ref{sec:prelim}. By Lemma \ref{lem:unbound}, $\tilde{E}$ is a closed relatively connected set and $\tilde{f} : \tilde{E} \to \R$ is an increasing quasisymmetric embedding.

\emph{Step 2}. We reduce the proof to the case that $E$ has no isolated points. If $E$ has isolated points, then define $\hat{E}$ and $\hat{f}$ as in Section \ref{sec:reduc}. Since $f(E) \subset \R$, then $\hat{f} : E \to \R$ and 
$\hat{f}$ is increasing. By Lemma \ref{lem:isolated}, $\hat{E}$ is a uniformly perfect closed set and $\hat{f}$ is quasisymmetric.

\emph{Step 3}. Let $I = (a,b)$ be a component of $\R\setminus E$. The bridge $\mathcal{B}(f(a),f(b))$ in this case is simply the interval $[f(a),f(b)]$. The mapping $f_I$ is defined as in Section \ref{sec:reflect}. The rest of the 
proof is similar to that of Theorem \ref{thm:main}.

\section{The quasisymmetric extension property in higher dimensions}\label{sec:examples}

This paper was motivated by the following question: given a uniformly perfect Cantor set $\mathcal{C}$ in $\R^n$ and a quasisymmetric mapping $f : \mathcal{C} \to \R^n$ that admits a homeomorphic extension on $\R^n$, is it always 
possible to extend $f$ quasisymmetrically in $\R^n$? While Theorem \ref{thm:main} shows that the answer is yes when $n= 1$, this is not the case when $n\geq 2$.
%the answer is no even when $\mathcal{C}$ is assumed to be uniformly perfect and $f$ is assumed bi-Lipschitz. 
In fact we show a slightly stronger statement.

\begin{theorem}
For any $n\geq 2$, there exists a compact, countable, relatively connected set $E \subset \R^n$ and a bi-Lipschitz mapping $f : E \to \R^n$ that admits a homeomorphic but no quasisymmetric extension on $\R^n$.
\end{theorem}

%Recall that a homeomorphism $f$ between two metric spaces is \emph{bi-Lipschitz} if both $f$ and $f^{-1}$ are $L$-Lipschitz for some $L>1$.

Before describing the construction we recall a definition. A domain $\Omega \subset \R^n$ is a \emph{$C$-John domain} if there exists $C\geq 1$ such that for any two points $x,y\in\Omega$, there is a path $\gamma \subset \Omega$ 
joining $x,y$ such that $\dist(z,\partial\Omega) \leq C^{-1}\min\{|x-z|,|y-z|\}$ for all $z\in\g$. In this case, the arc $\g$ is called a \emph{$C$-John arc}. It is a simple consequence of quasisymmetry that quasisymmetric images of 
John arcs are John arcs quantitatively. 

Fix now an integer $n\geq 2$ and define $h:\R^{n-1}\times\R$ with $h(v,t) = (v,2-t)$. Set $Q_0 = Q'_0 = [-1,1]^{n-1}\times[-1,1]$ and for each $k\in \N$ set 
\begin{align*}
Q_k &= [-4^{-k},4^{-k}]^{n-1}\times [2^{-k},2^{1-k}],
%\\
%Q'_k &= [-4^{-k},4^{-k}]^{n-1}\times [2-2^{1-k},2-2^{-k}]
\end{align*}
$h_k = h|Q_k$ and $Q_k' = h(Q_k)$. For $k=0$ we set $h_0 = \text{Id}$. Define
\[ U = \text{int}(Q_0 \setminus \bigcup_{k\in\N} Q_k) \text{ , } U' = \text{int}(Q'_0 \cup \bigcup_{k\in\N} Q'_k)\]
and $X = \partial U$, $X' = \partial U'$. Note that $U$ is a $C$-John domain for some $C\geq 1$.

For each integer $m\geq 0$ let $\zeta_m : \R^n \to \R^n$ be a similarity that maps $[-2,2]^n$ onto $[\frac{1}{2}4^{-m},4^{-m}]\times[-4^{-m-1},4^{-m-1}]^{n-1}$. For each $m,k\geq 0$ let $Q_{m,k}$, $Q'_{m,k}$, $U_m$, $U_m'$, $X_m$ 
and $X_m'$ be the images of $Q_k$, $Q_k'$, $U$, $U'$, $X$ and $X'$, respectively, under $\zeta_m$. Note that each $U_m$ is $C$-John domain.

For each $m,k\geq 0$ let $E_{m,k}$ be a finite set on $\partial Q_{m,k}\cap X_m$ such that 
\begin{equation}\label{eq:density}
\dist(x,E_{m,k})< 8^{-k-m} \text{ for all }x\in\partial Q_{m,k} \cap  X_m.
\end{equation} 
Let $P_m = \zeta_m(0,\dots,0,0)$, 
%$P_m' = \zeta_m(0,\dots,0,2)$, 
$P_m^* = \zeta_m(0,\dots,0,-1/2)$ and $P = (0,\dots,0)$. Set
\begin{align*}
E &= \{P\}\cup\{P_m,P_m^*\}_{m\geq 0}\cup\bigcup_{m,k\geq 0}E_{m,k}.
%E' &= \{P\}\cup\{P_m',P_m^*\}_{m\geq 0}\cup\bigcup_{m,k\geq 0}E'_{m,k}.
\end{align*}
Clearly, $E$ is compact and countable. Moreover, by choosing the sets $E_{m,k}$ to be relatively connected, we may assume that $E$ is relatively connected.

Define $f: E \to \R^n$ with $f(P) = P$, $f(P_m^*) = P_m^*$, $f(P_m) =\zeta_m(0,\dots,0,2)$ and
\[ f|E_{m,k} = \zeta_m\circ h_k\circ\zeta_m^{-1}|E_{m,k}.\]
Denote $E'_{m,k} = f(E_{m,k})$ and $E' = f(E)$. It is easy to show that $f$ is bi-Lipschitz and can be extended to a homeomorphism of $\R^n$. Let $F:\R^n \to \R^n$ be such an extension of $f$. We briefly describe why $F$ can not be 
quasisymmetric; the details are left to the reader.

Assume that $F$ is $\eta$-quasisymmetric. Fix $m\in\N$ to be chosen later. Let $x\in U_m$ with $\dist(x,X_m) = \dist(x,E_{m,k}) = 4^{-m}4^{-k}$. By quasisymmetry, (\ref{eq:density}) and the fact that $F|E_{m,k}$ is an isometry, its 
image $x'=F(x)$ satisfies $c_14^{-m}4^{-k} \leq\dist(x', E_{m,k}') \leq c_24^{-m}4^{-k}$ for some $0<c_1<c_2$ depending on $\eta$. We claim that if $m$ is chosen big enough, $x'\in U_m'$. Indeed, let $\g$ be a $C$-John arc 
connecting $x$ and $P_m^*$ in $\R^n\setminus E$. If $x' \in \R^n \setminus \overline{U_m'}$ then there would be a point $z\in F(\g) \cap X'_m$. If $z\in\partial Q'_{m,l}$ then 
$\dist(z,E'_{m,l}) \leq 8^{-m-l} < 2^{-m}\min\{|z-x'|,|z-P^*_m|\}$ which contradicts the quasisymmetry of $F$ if $m$ is sufficiently big. 

Let now $m$ be chosen as above. Let $x,y \in U_m$ with 
\[\dist(x,X_m) = \dist(x,E_{m,k}) = \dist(y,X_m) = \dist(y,E_{m,k}) = 4^{-m}4^{-k}\] 
and with $|x-y| = 4^{-m}2^{-k-1}$ where $k$ is chosen later. Let $a,b$ be the points in $E_{m,k}$ closest to $x,y$ respectively. By quasisymmetry of $F$, (\ref{eq:density}) and the fact that $F|E_{m,k}$ is an isometry, there exist 
constants $C_1,C_2>0$ depending only on $\eta$ such that the images $x',y'$ of $x,y$ satisfy $\dist(x',E_{m,k}'),\dist(y',E_{m,k}') \leq C_14^{-m}4^{-k}$ and $|x'-y'| \geq C_2 4^{-m}2^{-k}$. Let $\sigma$ be a $C$-John arc joining 
$x$ and $y$ in $\R^n \setminus E$. As before, we can show that $\sigma$ is contained in $U_m$ and its image $\sigma'$ is contained in $U'_m$. Let $z\in \sigma' \cap Q_{m,k}'$ such that $|z-x'| = |z-y'|$. Then, 
$\min\{|z-x'|,|z-y'|\} \geq \frac{1}{2}C_22^{-k}4^{-m}$ while $\dist(z,E_{m,k}') \leq \frac{1}{2}4^{-k}4^{-m}$
%However, for any curve $\tau$ joining $x'$ and $y'$ in $U_m$, there exists $z\in \tau$ 
%such that $\dist(z,E_{m,k}') $
%and $C_34^{-m}2^{-k} \leq \diam{\sigma'} \leq C_44^{-m}2^{-k}$ for some $C_3,C_4>0$ depending 
%only on $\eta$. Since $\sigma' \subset U'_m$, there exists some uniform constant $C_5>0$ such that $\dist(\sigma', E') \leq C_54^{-m}4^{-k} < C_62^{-k}\diam{\sigma'}$ 
and the John condition for $\sigma'$ fails if $k$ is sufficiently big. The latter contradicts the quasisymmetry of $F$.

\begin{remark}\label{rem:Cantor}
Let $\mathscr{C}$ be the standard ternary Cantor set in $[-\frac12,\frac12]$. If in the above construction we replace the finite sets $E_{m,k}$ by uniformly perfect Cantor sets $\mathcal{C}_{m,k}$ satisfying (\ref{eq:density}), and 
the points $P_m^*$ by sets $\mathcal{C}_m = \zeta_m(\mathscr{C}\times\{(0,\dots,0,\frac12)\})$, then we obtain a Cantor set
\[ \mathcal{C} = \{P\}\cup\{P_m\}_{m\geq 0}\cup\bigcup_{m\geq 0}\mathcal{C}_m\cup\bigcup_{m,k\geq 0}\mathcal{C}_{m,k},\]
for which the mapping $f$ defined as above is bi-Lipschitz and admits a homeomorphic extension on $\R^n$ but no quasisymmetric extension on $\R^n$.
\end{remark}

\bibliographystyle{amsplain}
\bibliography{proposal}

\end{document}